\documentclass{amsart}
\usepackage{amsmath}

\usepackage{amssymb}
\usepackage{amsthm}
\usepackage{amscd,amsbsy}
\usepackage{xcolor}
\usepackage[T1]{fontenc}

\newtheorem{theorem}{Theorem}[section]
\newtheorem{lemma}[theorem]{Lemma}

\newtheorem{corollary}[theorem]{Corollary}
\newtheorem{claim}[theorem]{Claim}

\theoremstyle{definition}
\newtheorem{definition}[theorem]{Definition}

\newtheorem{definitions and remarks}[theorem]{Definitions and Remarks}

\theoremstyle{remark}

\newtheorem{remark}[theorem]{Remark}
\newtheorem{remarks}[theorem]{Remarks}

\numberwithin{equation}{section}

\newcommand{\Jac}{\mathrm{Jac}\,}

\newcommand{\al}{{\alpha}}
\newcommand{\be}{{\beta}}

\newcommand{\De}{{\Delta}}
\newcommand{\ga}{{\gamma}}

\newcommand{\p}{{\partial}}
\newcommand{\s}{{\sigma}}

\newcommand{\vp}{{\varphi}}

\newcommand{\IN}{{\mathbb N}}

\newcommand{\IR}{{\mathbb R}}

\newcommand{\cC}{{\mathcal C}}

\newcommand{\cF}{{\mathcal F}}

\newcommand{\cO}{{\mathcal O}}
\newcommand{\cQ}{{\mathcal Q}}

\newcommand{\tbe}{{\tilde \be}}

\newcommand{\hf}{{\hat f}}

\newcommand{\hxi}{{\hat \xi}}

\newcommand{\hzeta}{{\hat \zeta}}
\newcommand{\hs}{{\hat \s}}
\newcommand{\hvp}{{\hat \vp}}
\newcommand{\hpsi}{{\hat \psi}}

\newcommand{\llb}{{[\![}}
\newcommand{\rrb}{{]\!]}}

\newcommand{\RN}[1]{%
  \textup{\uppercase\expandafter{\romannumeral#1}}%
}

\begin{document}
\title[Composite quasianalytic functions]{Composite quasianalytic functions}

\author[A.~Belotto]{Andr\'e Belotto da Silva}
\author[E.~Bierstone]{Edward Bierstone}
\author[M.~Chow]{Michael Chow}
\address[A.~Belotto]{Universit\'e Paul Sabatier, Institut de Math\'ematiques de Toulouse,
118 route de Narbonne, F-31062 Toulouse Cedex 9, France}
\email{andre.belotto\_da\_silva@math.univ-toulouse.fr}
\address[E.~Bierstone and M.~Chow]{University of Toronto, Department of Mathematics, 40 St. George Street,
Toronto, ON, Canada M5S 2E4}
\email{bierston@math.toronto.edu}
\email{mikey.chow@mail.utoronto.ca}
\thanks{Research supported by CIMI LabEx (Belotto), NSERC OGP0009070 (Bierstone), 
NSERC USRA and OGP0009070 (Chow)}



\begin{abstract}
We prove two main results on Denjoy-Carleman classes: (1) a composite function theorem which asserts 
that a function $f(x)$ in a quasianalytic Denjoy-Carleman class $\cQ_M$, which is formally composite with a generically submersive mapping $y=\vp(x)$ of class $\cQ_M$, at a single given point in the source (or in the target) of $\vp$ can be written locally as $f=g\circ\vp$,
where $g(y)$ belongs to a shifted Denjoy-Carleman class $\cQ_{M^{(p)}}$; (2) a statement on a similar loss of regularity 
for functions definable in the $o$-minimal
structure given by expansion of the real field by restricted functions of quasianalytic class $\cQ_M$. 
Both results depend on an estimate for the regularity of a $\cC^\infty$ solution $g$ of the equation $f=g\circ\vp$, with $f$ and $\vp$ as above. The composite function result depends also on a \emph{quasianalytic continuation} theorem, 
which shows that the formal assumption at a given point in (1) propagates to a formal composition condition
at every point in a neighbourhood. 
\end{abstract}

\date{\today}
\maketitle
\setcounter{tocdepth}{1}
\tableofcontents

\section{Introduction}\label{sec:intro}

This article contains two main results on Denjoy-Carleman classes: 

\smallskip\noindent
(1) a composite function theorem which asserts that a function $f(x)$ in a quasianalytic Denjoy-Carleman class $\cQ_M$, which is formally composite with a generically submersive mapping $y=\vp(x)$ of class $\cQ_M$, at a given point
$x=a$ (or at a given point $y=b$) can be written locally as $f=g\circ\vp$,
where $g(y)$ belongs to a shifted Denjoy-Carleman class $\cQ_{M^{(p)}}$ (see Theorems \ref{thm:source} and \ref{thm:target} for 
precise versions of this
assertion that are local at $a$ or $b$, respectively); 

\smallskip\noindent
(2) a statement on a similar loss of regularity for functions definable in the $o$-minimal
structure given by expansion of the real field by restricted functions of quasianalytic class $\cQ_M$ (Theorem \ref{thm:definable}).

\smallskip
Both results depend on an estimate for the regularity of a $\cC^\infty$ solution $g$ of the equation $f=g\circ\vp$, with $f$ and $\vp$ as above
(Theorem \ref{thm:estimate}). The composite function result depends also on a \emph{quasianalytic continuation} theorem (Theorem
\ref{thm:contin} below), which shows that the formal assumption at a given point in (1) propagates to a formal composition condition
at every point in a neighbourhood. 

Quasianalytic Denjoy-Carleman classes go back to E. Borel \cite{Borel} and were characterized (following questions of Hadamard in
studies of linear partial differential equations \cite{Had}) by the Denjoy-Carleman theorem \cite{Den}, \cite{Carl}. Quasianalytic classes
arise in model theory as the classes of $\cC^\infty$ functions that are definable in a given polynomially-bounded $o$-minimal
structure \cite{Mil}, \cite{RSW}.
Relevant background on Denjoy-Carleman and quasianalytic classes is presented in Section \ref{sec:quasian} below. 

Given a $\cC^\infty$ function $f(x)$ or a $\cC^\infty$
mapping $y=\vp(x)$ defined on an open subset of $\IR^n$, we write $\hf_a$ or $\hvp_a$ for the formal Taylor expansions at a point $x=a$.

\begin{theorem}\label{thm:source}
Let $\cQ_M$ denote a quasianalytic Denjoy-Carleman class, and let $\vp: V \to W$ be a mapping of class $\cQ_M$ between 
open sets $V\subset \IR^m$ and $W\subset \IR^n$, which is generically a submersion (i.e., generically of rank $n$). Let $f \in \cQ_M(V)$. 
If $a \in V$ and $\hf_a = G\circ\hvp_a$, where $G$ is a formal power series centred at $\vp(a)$, then there is a relatively compact 
neighbourhood $U$ of $a$ in $V$ and a function $g\in \cQ_{M^{(p)}}(\vp(\overline{U}))$, for some $p\in \IN$, such that $f=g\circ\vp$ on $U$.
\end{theorem}

\begin{theorem}\label{thm:target}
Let $\cQ_M$ denote a quasianalytic Denjoy-Carleman class, and let $\vp: V \to W$ be a $\cQ_M$-mapping of  
$\cQ_M$-manifolds which is proper and generically submersive. Let $f \in \cQ_M(V)$ and let $b\in \vp(V)$.
Suppose that $\hf_a = G\circ\hvp_a$, for all $a\in \vp^{-1}(b)$, where $G$ is a formal power series centred at $b$. 
Then, after perhaps shrinking $W$ to an appropriate neighbourhood of $b$,
there exists $g\in \cQ_{M^{(p)}}(\vp(V))$, for some $p$, such that $f=g\circ\vp$.
\end{theorem}

The formal expansions in Theorem \ref{thm:target} can be defined using local coordinate systems, and the condition
$\hf_a = G\circ\hvp_a$ is independent of the choice of local coordinates at $a$ and $b$. 

We recall that in Glaeser's $\cC^\infty$ composition theorem \cite{G}, where $\vp$ is real-analytic, $f$ is $\cC^\infty$ and
we seek a $\cC^\infty$ solution $g$, it is necessary to assume the existence of a formal solution at every point of 
the target. It is striking that, as a result of quasianalytic continuation, 
it is sufficient to assume there is a formal solution at a single point in Theorems \ref{thm:source} and \ref{thm:target}. In particular, 
Theorem \ref{thm:source} reduces to the classical statement in the real-analytic case (cf. \cite[Lemma 7.8]{BMihes}, \cite{Malg}), 
even though it seems unlikely that the ring of formal power series at a point
is flat over the local ring of germs of functions of class $\cQ_M$, in general.

The main ingredients in our proofs of Theorems \ref{thm:source}, \ref{thm:target} are 
Theorems \ref{thm:contin} and \ref{thm:estimate}, and a quasianalytic version \cite{Ncomp} of 
Glaeser's theorem \cite{G}. It follows from Theorem \ref{thm:contin} that Theorems \ref{thm:source} and 
\ref{thm:target} are equivalent (see Remark \ref{rem:comp}) and also that, in Theorem \ref{thm:target}, we can weaken the assumption to 
$\hf_a = G\circ\hvp_a$, for a single point $a$ in each connected component of $\vp^{-1}(b)$.

\subsection{Quasianalytic continuation}\label{subsec:contin}
The following theorem is a further development of quasianalytic continuation techniques introduced in  \cite[Sect.\,4]{BBB}.

\begin{theorem}\label{thm:contin}
Let $\cQ$ denote a quasianalytic class (see Definition \ref{def:quasian}).
Let $\vp: V \to W$ denote a $\cQ$-mapping, where $V$ is a $\cQ$-manifold and $W$ is an open neighbourhood of $0\in \IR^n$.
Let $f \in \cQ(V)$ and let $H$ denote a formal power series at $0\in \IR^n$. Then:
\begin{enumerate}
\item $\displaystyle{\{a \in \vp^{-1}(0): \hf_a = H\circ\hvp_a\}}$ 
is open and closed in $\vp^{-1}(0)$.
\smallskip
\item Suppose that $\vp$ is proper and generically of rank $n$.
Assume $\hf_a = H\circ\hs_a$, for all $a \in \vp^{-1}(0)$. Then, after perhaps shrinking $W$, 
$f$ is \emph{formally composite} with $\vp$; i.e.,
for all $b \in W$, there exists a power series $H_b$ centred at $b$ such that $\hf_a = H_b\circ\hvp^*_a$, for all $a \in \vp^{-1}(b)$.
\end{enumerate}
\end{theorem}

The composite function theorem \ref{thm:target} was proved by Chaumat and Chollet \cite{CC} in the special case that $\vp$ is a 
real-analytic mapping, under the stronger assumption that $f$ is formally composite with $\vp$ as in the conclusion 
of Theorem \ref{thm:contin}(2)
(i.e., throughout the image of $\vp$, as in Glaeser's theorem).
Theorem \ref{thm:estimate} below is a generalization to a mapping in a given Denjoy-Carleman class of an estimate for a real-analytic
mapping in \cite[Sect.\,III, Prop.\,8]{CC}.

\subsection{Regularity estimates}\label{subsec:est} Theorem \ref{thm:estimate} and Corollary \ref{cor:estimate}
following will be proved in Section 4. Our proof of Theorem \ref{thm:estimate} follows
the reasoning used in \cite[Lemmas 3.1,\,3.4]{BBB}, which give more precise estimates in the special cases that $\s$ is 
a power substitution or a blowing-up. See Section 2 for the notation used in the theorem.

\begin{theorem}\label{thm:estimate} 
Let $\cQ_M$ denote a Denjoy-Carleman class (see Definition \ref{def:DC}). Let $\s: V \to W$ denote a mapping of
class $\cQ_M$ between open subsets of $\IR^n$, such that the Jacobian determinant $\det (\p \s / \p x)$ is a monomial
$x^\ga = x_1^{\ga_1}\cdots x_n^{\ga_n}$ times a nowhere vanishing factor. Let $g \in \cC^\infty(W)$ and assume that
$f := g\circ \s \in \cQ_M(V)$. Then, for every compact $K\subset V$, $g$ is of class $\cQ_{M^{(p)}}$ on $\s(K)$, where
\begin{equation}\label{eq:DCest}
p := 2|\ga| +1;
\end{equation}
i.e., there exist constants $A,B > 0$ such that
$$
\left|\frac{\p^{|\be|}g}{\p y^{\be}}\right| \leq A B^{|\be|} \be! M_{p|\be|}\quad \text{ on } \s(K), \text{ for every } \be \in \IN^n.
$$
Moreover, if $\cQ_M$ is closed under differentiation (in particular, if $\cQ_M$ is a quasianalytic Denjoy-Carleman class;
see Section \ref{sec:quasian}),
then we can take
\begin{equation}\label{eq:quasianDCest}
p = |\ga| +1.
\end{equation}
\end{theorem} 

\begin{corollary}\label{cor:estimate}
Let $\cQ_M$ denote a quasianalytic Denjoy-Carleman class, and let $\vp: V \to W$ be a proper (or semiproper),
generically submersive mapping of class $\cQ_M$, 
where $V$ and $W$ are $\cQ_M$-manifolds.  Let $g \in \cC^\infty(W)$. 
If $f := g \circ \vp \in \cQ_M(V)$, then, for every 
relatively compact open $U\subset W$, there exists $p\in \IN$ (depending only on $\vp$ and $U$) such that $g \in \cQ_{M^{(p)}}(\vp(V)\bigcap U)$.
\end{corollary}

\subsection{Model theory}\label{subsec:model}
Let $\cQ_M$ denote a quasianalytic Denjoy-Carleman class, and let $\IR_{\cQ_M}$ denote the expansion of the real field by
restricted functions of class $\cQ_M$ (i.e., restrictions to closed cubes of $\cQ_M$-functions, extended by $0$ outside the cube).
Then $\IR_{\cQ_M}$ is an $o$-minimal structure, and $\IR_{\cQ_M}$ is both polynomially bounded and model-complete \cite{RSW}. 

By \cite{Mil}, the class of $\cC^\infty$ functions
that are definable in any given polynomially bounded $o$-minimal structure satisfy the quasianalyticity property (Definition \ref{def:quasian}(3)). 
The following will be proved in Section 3, using the regularity estimate above.

\begin{theorem}\label{thm:definable}
Let $\cQ_M$ denote a quasianalytic Denjoy-Carleman class. If $f\in \cC^\infty(W)$, where $W$ is open in $\IR^n$, and $f$ is definable
in $\IR_{\cQ_M}$, then there exists $p \in \IN$ such that $f\in \cQ_{M^{(p)}}(W)$.
\end{theorem}

The graph of a function definable in $\IR_{\cQ_M}$ is \emph{sub-quasianalytic} in the obvious sense generalizing 
subanalytic. If $f\in\cC^\infty(W)$
has sub-quasianalytic graph, 
then $f \in \cQ_{M^{(p)}}(U)$, for some $p=p_U$, for every relatively compact open $U\subset W$;
in general, however, $\{p_U\}$ need not be bounded.

\section{Quasianalytic classes}\label{sec:quasian}

We consider a class of functions $\cQ$ given by the association, to every 
open subset $U\subset \IR^n$, of a subalgebra $\cQ(U)$ of $\cC^\infty (U)$ containing
the restrictions to $U$ of polynomial functions on $\IR^n$, and closed under composition 
with a $\cQ$-mapping (i.e., a mapping whose components belong to $\cQ$). 
We assume that $\cQ$ determines a sheaf of local $\IR$-algebras of $\cC^\infty$ functions on $\IR^n$,
for each $n$, which we also denote $\cQ$.

\begin{definition}[quasianalytic classes]\label{def:quasian}
We say that $\cQ$ is \emph{quasianalytic} if it satisfies the following three axioms:

\begin{enumerate}
\item \emph{Closure under division by a coordinate.} If $f \in \cQ(U)$ and
$$
f(x_1,\dots, x_{i-1}, a, x_{i+1},\ldots, x_n) = 0,
$$
where $a \in \IR$,  then $f(x) = (x_i - a)h(x),$ where $h \in \cQ(U)$.

\smallskip
\item \emph{Closure under inverse.} Let $\varphi : U \to V$
denote a $\cQ$-mapping between open subsets $U$, $V$ of $\IR^n$.
Let $a \in  U$ and suppose that the Jacobian matrix
$(\partial \varphi/\partial x)(a)$ is invertible. Then there are neighbourhoods $U'$ of $a$ and $V'$ of 
$b := \varphi(a)$, and a $\cQ$-mapping  $\psi: V' \to U'$ such that
$\psi(b) = a$ and $\psi\circ \varphi$  is the identity mapping of
$U '$.

\smallskip
\item \emph{Quasianalyticity.} If $f \in \cQ(U)$ has Taylor expansion zero
at $a \in U$, then $f$ is identically zero near $a$.
\end{enumerate}
\end{definition}

\begin{remarks}\label{rem:axioms} (1)\, Axiom \ref{def:quasian}(1) implies that, 
if $f \in \cQ(U)$, then all partial derivatives of $f$ belong to $\cQ(U)$. 

\smallskip\noindent
(2)\, Axiom \ref{def:quasian}(2) is equivalent to the property that the implicit function theorem holds for functions of 
class $\cQ$.  It implies that the reciprocal of a nonvanishing function of class $\cQ$ is also of class $\cQ$.

\smallskip\noindent
(3)\, Our two main examples of quasianalytic classes are quasianalytic Denjoy-Carleman classes (see \S\ref{subsec:DC}),
and the class of $\cC^\infty$ functions definable in a given polynomially bounded
$o$-minimal structure (\S\ref{subsec:model}). In the latter case, we can define a quasianalytic class $\cQ$ in the axiomatic
framework above by taking $\cQ(U)$ as the subring of $\cC^\infty(U)$ of functions
$f$ such that $f$ is definable in some neighbourhood of any point of $U$ (or, equivalently,
such that $f|_V$ is definable, for every relatively compact definable open $V\subset U$); the division and inverse properties
are immediate from definability and the corresponding $\cC^\infty$ assertions.
\end{remarks} 

The elements of a quasianalytic class $\cQ$ will be called \emph{quasianalytic functions}. 
A category of manifolds and mappings of class $\cQ$ can be defined in a standard way. The category 
of $\cQ$-manifolds is closed under blowing up with centre a $\cQ$-submanifold \cite{BMselecta}.

Resolution of singularities holds in a quasianalytic class \cite{BMinv}, \cite{BMselecta}. Resolution of
singularities of an ideal does not require that the ideal be finitely generated; see \cite[Thm.\,3.1]{BMV}.

\subsection{Quasianalytic Denjoy-Carleman classes}\label{subsec:DC}
We use standard multiindex notation: Let $\IN$ denote the nonnegative integers. If $\al = (\al_1,\ldots,\al_n) \in
\IN^n$, we write $|\al| := \al_1 +\cdots +\al_n$, $\al! := \al_1!\cdots\al_n!$, $x^\al := x_1^{\al_1}\cdots x_n^{\al_n}$,
and $\p^{|\al|} / \p x^{\al} := \p^{\al_1 +\cdots +\al_n} / \p x_1^{\al_1}\cdots \p x_n^{\al_n}$. We write $(i)$ for the
multiindex with $1$ in the $i$th place and $0$ elsewhere.

\begin{definition}[Denjoy-Carleman classes]\label{def:DC}
Let $M = (M_k)_{k\in \IN}$ denote a sequence of positive real numbers which is \emph{logarithmically
convex}; i.e., the sequence $(M_{k+1} / M_k)$ is nondecreasing.
A \emph{Denjoy-Carleman
class} $\cQ = \cQ_M$ is a class of $\cC^\infty$ functions determined by the following condition: A function 
$f \in \cC^\infty(U)$ (where $U$ is open in $\IR^n$) is of class $\cQ_M$ if, for every compact subset $K$ of $U$,
there exist constants $A,\,B > 0$ such that
\begin{equation}\label{eq:DC}
\left|\frac{\p^{|\al|}f}{\p x^{\al}}\right| \leq A B^{|\al|} \al! M_{|\al|}
\end{equation}
on $K$, for every $\al \in \IN^n$.
\end{definition}

\begin{remark}\label{rem:DC}
The logarithmic convexity assumption implies that
$M_jM_k \leq M_0M_{j+k}$, for all $j,k$, and that
the sequence $((M_k/M_0)^{1/k})$ is nondecreasing.
The first of these conditions guarantees that $\cQ_M(U)$ is a ring, and 
the second that $\cQ_M(U)$ contains the ring $\cO(U)$ of real-analytic functions on $U$,
for every open $U\subset \IR^n$.
(If $M_k=1$, for all $k$, then $\cQ_M = \cO$.)
\end{remark}

If $X$ is a closed subset of $U$, then $\cQ_M(X)$ will denote the ring of restrictions to $X$
of $\cC^\infty$ functions which satisfy estimates of 
the form \eqref{eq:DC}, for every compact $K\subset X$.

A Denjoy-Carleman class $\cQ_M$ is a quasianalytic class in the sense of Definition \ref{def:quasian}
if and only if the sequence
$M = (M_k)_{k\in \IN}$ satisfies the following two assumptions in addition to those
of Definition \ref{def:DC}.
\begin{enumerate}
\item[(a)] $\displaystyle{\sup \left(\frac{M_{k+1}}{M_k}\right)^{1/k} < \infty}$.

\smallskip
\item[(b)] $\displaystyle{\sum_{k=0}^\infty\frac{M_k}{(k+1)M_{k+1}} = \infty}$.
\end{enumerate}

It is easy to see that the assumption (a) implies that $\cQ_M$ is closed under differentiation.
The converse of this statement is due to S. Mandelbrojt \cite{Mandel}. In a Denjoy-Carleman class
$\cQ_M$, closure under differentiation is equivalent to the axiom \ref{def:quasian}(1) of closure under division by a
coordinate---the converse of Remark \ref{rem:axioms}(1) is a consequence of the fundamental
theorem of calculus:
\begin{equation}\label{eq:FTC}
f(x_1,\ldots,x_n) - f(x_1,\ldots, 0,\ldots, x_n) = x_i\int_0^1\frac{\p f}{\p x_i}(x_1,\ldots,tx_i,\ldots,x_n) dt
\end{equation}
(where $0$ in the left-hand side is in the $i$th place).

According to the Denjoy-Carleman theorem, the class $\cQ_M$ is quasianalytic (axiom \ref{def:quasian}(3))
if and only if the assumption (b) holds \cite[Thm.\,1.3.8]{Horm}.

Closure of a Denjoy-Carleman class $\cQ_M$ under composition is due to Roumieu \cite{Rou} and closure under
inverse to Komatsu \cite{Kom}; see \cite{BMselecta} for simple proofs. A Denjoy-Carleman class $\cQ_M$ satisfying 
the assumptions (a) and (b)
above is thus a quasianalytic class, in the sense of Definition \ref{def:quasian}.

If $\cQ_M$, $\cQ_N$ are Denjoy-Carleman classes, then $\cQ_M(U) \subseteq \cQ_N(U)$, for all $U$,
if and only if $\sup \left(M_k /N_k\right)^{1/k} < \infty$ (see \cite[\S1.4]{Th1}); in this case, we write
$\cQ_M \subseteq \cQ_N$. For any given Denjoy-Carleman class $\cQ_M$, there is a function in $\cQ_M((0,1))$
which is nowhere in any given smaller class \cite[Thm.\,1.1]{Jaffe}

\subsection{Shifted Denjoy-Carleman classes}\label{subsec:shift}
Given $M = (M_j)_{j\in \IN}$ and a positive integer $p$, let $M^{(p)}$ denote the sequence $M^{(p)}_j := M_{pj}$.

If $M$ is logarithmically convex, then $M^{(p)}$ is logarithmically convex:
$$
\frac{M_{kp}}{M_{(k-1)p}} = \frac{M_{kp}}{M_{kp-1}} \cdots \frac{M_{kp -p+1}}{M_{kp - p}} \leq
\frac{M_{kp+p}}{M_{kp+p-1}} \cdots \frac{M_{kp+1}}{M_{kp}} = \frac{M_{(k+1)p}}{M_{kp}}.
$$
Therefore, if $\cQ_M$ is a Denjoy-Carleman class, then so is $\cQ_{M^{(p)}}$. Clearly,
$\cQ_M \subseteqq \cQ_{M^{(p)}}$. Moreover, the assumption (a)
above for $\cQ_M$ immediately implies the same condition for $\cQ_{M^{(p)}}$. In general, however, it is not true 
that assumption (b) (i.e., the quasianalyticity axiom (3)) for $\cQ_M$ implies (b) for $\cQ_{M^{(p)}}$ \cite[Example 6.6]{Nelim}.

In particular, in general,
$\cQ_{M^{(p)}} \supsetneq \cQ_M$. Moreover, $\cQ_{M^{(2)}}$ is the smallest Denjoy-Carleman class containing
all $g \in \cC^\infty(\IR)$ such that $g(t^2)\in \cQ_M(\IR)$ \cite[Rmk.\,6.2]{Nelim}.

\section{Quasianalytic continuation}\label{sec:contin}
The goal of this section is to prove Theorem \ref{thm:contin}. We begin with three lemmas that strengthen
results in \cite[Sect.\,4]{BBB}.

Let $\cF_a$ denote the ring of formal power series centred at a point $a\in \IR^m$; thus 
$\cF_a \cong \IR\llb x_1,\ldots,x_m\rrb$. If $U$ is open in $\IR^m$ and $f \in \cC^\infty(U)$,
then $\hf_a \in \cF_a$ denotes the formal Taylor expansion of $f$ at a point $a \in U$; i.e.,
$\hf_a(x) = \sum_{\al\in\IN^m}(\p^{|\al|}f/\p x^{\al})(a)x^\al/\al!$ (likewise
for a $\cC^\infty$ mapping $U \to \IR^n$).

Let $\cQ$ denote a quasianalytic class (Definition \ref{def:quasian}).

\begin{lemma}\label{lem:contin}
Let $V,\,W$ denote open neighbourhoods of the origin in $\IR^m$, with coordinate systems
$x= (x_1,\ldots,x_m),\, y=(y_1,\ldots,y_m)$, respectively. (Assume $V$ is chosen so that every
coordinate hyperplane $(x_i = 0)$ is connected). Let $\psi:V \to W$ denote a $\cQ$-mapping such
that the Jacobian determinant $\det (\p \psi / \p x)$ is a monomial times
an invertible factor in $\cQ(V)$. Let $f \in \cQ(V)$ and let $H \in \cF_0$ be a formal power series
centred at $0\in W$, such that $\hf_0 = H\circ\hpsi_0$. Then, for all $\be \in \IN^m$, there exists
$f_\be \in \cQ(V)$ such that $f_0 = f$ and
\begin{enumerate}
\item for all $a\in\psi^{-1}(0)$ and all $\be\in \IN^m$, 
\begin{equation}\label{eq:cont}
\hf_{\be,a} = \frac{\p^{|\be|} H}{\p y^\be}\circ \hpsi_a;
\end{equation}
\item for all $a\in V$, $\hf_a = H_a\circ \hpsi_a$, where $H_a \in \cF_{\psi(a)}$ denotes the formal power series
\begin{equation}\label{eq:contin}
H_a := \sum_{\be \in \IN^m} \frac{f_\be(a)}{\be !}y^\be;
\end{equation}
\item each $f_\be$, $\be \in \IN^m$, and therefore also $H_a \in \cF_{\psi(a)}$ (as a function of $a$)
is constant on connected components of the fibres of $\psi$;
\smallskip
\item if $H$ is independent of some variable $y_j$, then $H_a$ is independent of $y_j$, for all $a\in V$.
\end{enumerate}
\end{lemma}

\begin{proof}
The lemma with items (2),\,(3) is a restatement of \cite[Thm.\,4.1]{BBB} 
(the proof of (3) in the latter, in fact, uses a simpler version \cite[Lemma 4.2]{BBB} of 
Lemma \ref{lem:nowak} below) and item (1) is contained in the proof of \cite[Thm.\,4.1]{BBB}.

To prove (4):  $H$ is independent of some variable $y_j$ if and only if $\p^{|\be|} H / \p y^\be = 0$ whenever
$\be = (\be_1,\ldots,\be_m)$ with $\be_j\neq 0$. By \eqref{eq:cont}, the latter condition implies that
that $\hf_{\be,0} = 0$ whenever $\be_j \neq 0$; i.e, that $f_\be = 0$ whenever $\be_j \neq 0$ (by quasianalyticity).
Therefore, if $H$ is independent of $y_j$, then, for every $a$, $f_\be(a)=0$ whenever $\be_j\neq 0$; i.e., $H_a$ is
independent of $y_j$, by \eqref{eq:contin}.
\end{proof}

\begin{lemma}\label{lem:nowak}
Let $\vp: V \to W$ denote a $\cQ$-mapping, where $W$ is a neighbourhood of $0$ in $\IR^n$
and $V$ is a $\cQ$-manifold. Let $f_\be \in \cQ(V)$, for all $\be\in\IN^n$. Assume that
\begin{enumerate}
\item 
all $f_\be$ are constant on the fibre $\vp^{-1}(0)$; i.e., 
\begin{equation*}
H_a := \sum_{\be \in \IN^n} \frac{f_\be(a)}{\be !}y^\be,\quad a \in \vp^{-1}(0),
\end{equation*}
is a formal power series $H_a = H$ independent of $a$;
\smallskip
\item
for all $a\in\vp^{-1}(0)$ and $\be\in \IN^m$,
\begin{equation*}
\hf_{\be,a} = \frac{\p^{|\be|} H}{\p y^\be}\circ \hvp_a.
\end{equation*}
\end{enumerate}
If $K$ is a compact subset of $\vp^{-1}(0)$, then all $f_\be$ are constant on all fibres of $\vp$
in some neighbourhood of $K$.
\end{lemma}

\begin{proof}[Proof \emph{(cf. proofs of \cite[Lemma 4.2, Cor.\,4.5]{BBB})}]
For every $k \in \IN$, let 
$$
P_k := \{(\xi,\eta,\zeta) \in V\times V\times W: \vp(\xi) = \zeta = \vp(\eta),\, f_\be(\xi) = f_\be(\eta),\, |\be|\leq k\}.
$$
Then the decreasing sequence of closed quasianalytic sets $P_0 \supset P_1 \supset P_2 \supset \cdots$
stabilizes in some neighbourhood of $K$, by
topological noetherianity \cite[Thm.\,6.1]{BMselecta}; say, $P_k = P_{k_0}$, $k\geq k_0$, in such a
neighbourhood. It follows that, if $f_\be$ is
constant on the fibres of $\vp$ in a given neighbourhood of $K$, for all $\be \leq k_0$, then $f_\be$ is
constant on the fibres of $\vp$ in some neighbourhood of $K$, for all $\be$.

Therefore, it is enough to prove the following assertion: given $\be \in \IN^n$, $f_\be$ is constant on the fibres of $\vp$ over $W$, in some neighbourhood of $K$. The following argument is due to Nowak \cite{Ndiv}.
Define 
\begin{equation*}
P:= \{(\xi, \eta, \zeta) \in V \times V \times W: \vp(\xi) = \zeta = \vp(\eta),\, f_\be(\xi) \neq f_\be(\eta)\}.
\end{equation*}

Suppose the assertion is false. Then
there is a point $(a_1,a_2,0) \in \overline{P}$ and,
by the quasianalytic curve selection
lemma (see \cite[Thm.\,6.2]{BMselecta}), a quasianalytic arc 
$(\xi(t), \eta(t), \zeta(t)) \in V\times V\times W$
such that $(\xi(0), \eta(0), \zeta(0)) = (a_1,a_2,0)$ and $(\xi(t), \eta(t), \zeta(t)) \in P$ if $t\neq 0$. Then
\begin{equation*}
(f_\be\circ\xi)^{\wedge}_0 = \hf_{\be,a_1}\circ\hxi_0 = \frac{\p^{|\be|}H}{\p y^{\be}}\circ\hvp_{a_1}\circ\hxi_0
= \frac{\p^{|\be|}H}{\p y^{\be}}\circ(\vp\circ\xi)^{\wedge}_0 = \frac{\p^{|\be|}H}{\p y^{\be}}\circ\hzeta_0.
\end{equation*}
Likewise, $(f_\be\circ\eta)^{\wedge}_0 = (\p^{|\be|}H / \p y^{\be})\circ\hzeta_0$, so that 
$(f_\be\circ\xi)^{\wedge}_0 = (f_\be\circ\eta)^{\wedge}_0$. Since
$f_\be\circ\xi,\, f_\be\circ\eta$ are quasianalytic functions of $t$, $f_\be\circ\xi = f_\be\circ\eta$; a contradiction.
\end{proof}

\begin{lemma}\label{lem:glue}
Let $\vp: V \to W$ denote a $\cQ$-mapping, where $W$ is a neighbourhood of $0$ in $\IR^n$
and $V$ is a $\cQ$-manifold. Let $K$ be a compact subset of $V$ 
and let $\{U\}$ be an open covering of $\vp^{-1}(0)\bigcap K$.
Let $H \in \cF_0$ be a power series centred at $0 \in W$. Suppose there exists 
$f_U \in \cQ(U)$, for each $U$, such that $\hf_{U,a} = \hvp_{a}^*(H)$, for all $a \in \vp^{-1}(0)\cap U$. 
Then there exists $f \in \cQ(V')$, where $V'$ is a neighbourhood of $\vp^{-1}(0)\bigcap K$,
such that $\hf_a = \hvp^*_a(H)$, for all $a \in \vp^{-1}(0)\bigcap K$.
\end{lemma}

\begin{proof} Lemma \ref{lem:glue} generalizes \cite[Lemma 4.4]{BBB}
and the proof is the same.
\end{proof}

\begin{proof}[Proof of Theorem \ref{thm:contin}] 
(1) is a restatement of \cite[Prop.\,4.6]{BBB}, so we only prove (2) here.

Write $\vp = (\vp_1,\ldots,\vp_n)$. Every point of $\vp^{-1}(0)$ has a neighbourhood $U$
in $V$ with a coordinate system $(x_1,\ldots,x_m)$ in which the Jacobian submatrix
$$
\frac{\p (\vp_1,\ldots,\vp_n)}{\p (x_1,\ldots,x_n)}
$$
is generically of rank $n$.

Fix such $U$ and define $\psi_U: U\to \IR^m$ by
\begin{equation}\label{eq:ext}
\psi_U(x_1,\ldots,x_m) = (\vp(x), x_{n+1},\ldots,x_m).
\end{equation}
Then, for all $a\in \vp^{-1}(0)$,
$$
\hf_a = H\circ \hvp_a = H_U\circ \hpsi_{U,a},
$$
where $H_U(y_1,\ldots,y_m) = H(y_1,\ldots,y_n)$; in particular, the formal
power series $H_U(y)$ is independent of $y_{n+1},\ldots,y_m$.

We claim there is a proper surjective mapping $\s_U: Z_U \to U$, where $Z_U$ is a $\cQ$-manifold
of dimension $m$ ($\s_U$ is a composite of finitely many blowings-up if $U$ is small enough)
such that every point of $Z_U$ has a neighbourhood with coordinates $(z_1,\ldots,z_m)$ in
which the Jacobian determinant $\det (\p (\psi_U\circ\s_U)/\p z)$ is a monomial in $z$
times an invertible factor. In fact, by resolution of singularities \cite[Thm.\,5.9]{BMselecta},
there exists $\s_U: Z_U \to U$ such that every point of $Z_U$ has a neighbourhood with
coordinates $(z_1,\ldots,z_m)$ in which $(\det(\p \psi_U/\p x))\circ\s_U$ and
$\det(\p \s_U/\p z)$ are both monomials times invertible factors, so the claim follows from
the chain rule.

Now, $\vp^{-1}(0)$ is covered by finitely many open sets $U$ as above. Choose compact $K_U \subset U$ so that
$\vp^{-1}(0) = \bigcup K_U$. Consider
the mapping $\s: Z\to V$ from the disjoint union $Z := \coprod Z_U$, 
where $\s$ is given by $\s_U$ on each $Z_U$. Set $\Phi := \vp\circ\s$.
By Lemmas \ref{lem:contin}, \ref{lem:nowak} and \ref{lem:glue}, after shrinking $W$ and the $Z_U$ to suitable neighbourhoods
of $0$ and the $\s_U^{-1}(K_U)$ (respectively), $\s(Z) = \vp^{-1}(W)$ and
there exist $f_\be \in \cQ(Z)$, $\be \in \IN^n$, such that each $f_\be$ is
constant on the fibres of $\Phi$ and, for all $U$ and $a'\in Z_U$, 
\begin{equation}\label{eq:compos}
(f\circ\s_U)^{\wedge}_{a'} = H_{a'}\circ\hvp_{\s_U(a')}\circ\hs_{U,a'},
\end{equation}
where
$$
H_{a'} = \sum_{\be\in\IN^n}\frac{f_\be(a')}{\be!}y^\be
$$
($\hs_{U,a'}$ can be defined using local coordinates $(z_1,\ldots,z_m)$ in a
neighbourhood of any point of $Z_U$).
So we can set $H_b := H_{a'}$, where $a'\in \Phi^{-1}(b)$, for any $b\in W$.

Finally, we note that the power series homomorphism $\cF_{\vp(a')} \to \cF_{a'}$
given by composition with $\hs_{U,a'}$ is injective, for any $a'\in Z_U$ as above.
This is a consequence of the fact that, with respect to local coordinates $(z_1,\ldots,z_m)$
at $a'$, the determinant of the Jacobian matrix of \emph{formal derivatives} $\p \hs_{U,a'} /\p z$
is a nonzero formal power series, by the quasianalyticity axiom and the fact that $\s_U$ is
generically of rank $m$. We can therefore conclude from \eqref{eq:compos} that 
$\hf_a = H_b\circ \hvp_a$, for all $b\in W$ and $a\in \vp^{-1}(b)$, as required.
\end{proof}

\section{Proofs of the main theorems}\label{sec:proofs}

In this section, we prove Theorems \ref{thm:source}, \ref{thm:target} and \ref{thm:definable}, using
the regularity
estimates of Theorem \ref{thm:estimate} and Corollary \ref{cor:estimate}, which are proved in Section \ref{sec:estimate}. For the composite function theorems we also need the quasianalytic continuation
theorem \ref{thm:contin}.

\subsection{Composite function theorems}\label{subsec:comp} 
We can assume that $W$ is an open neighbourhood of $0\in \IR^n$, in both Theorems \ref{thm:source} and
\ref{thm:target}.

\begin{remark}\label{rem:comp} \emph{Theorems \ref{thm:source} and \ref{thm:target} are equivalent} 
(for given $\dim V \geq n$):
With the hypotheses of Theorem \ref{thm:source}, it follows from the quasianalytic continuation theorem \ref{thm:contin} that
there is a neighbourhood of $a$ on which $f$ is formally composite with $\s$. It is then easy to see that Theorem \ref{thm:target}
implies \ref{thm:source}. Conversely, assuming Theorem \ref{thm:source}, it follows from the hypotheses of Theorem \ref{thm:target}
that each $a\in \vp^{-1}(0)$ has a relatively compact neighbourhood $U$ on which $f=g\circ\vp$, where 
$g\in \cQ_{M^{(p)}}(\vp(\overline{U}))$, 
for some $p$, and the conclusion of Theorem \ref{thm:target} is a simple consequence.
\end{remark}

\begin{proof}[Proof of Theorem \ref{thm:target}]
It follows from Theorem \ref{thm:contin} that, after
shrinking $W$ to a suitable neighbourhood of $b$, $f$ is formally composite with $\vp$. 
By a quasianalytic version \cite{Ncomp}
of Glaeser's theorem \cite{G}, $f=g\circ\vp$, where $g\in \cC^\infty(W)$. The theorem then follows from Theorem \ref{thm:estimate}
or Corollary \ref{cor:estimate}.
\end{proof}

\begin{remark}\label{rem:semiproper}
Theorems \ref{thm:target} and \ref{thm:contin} do not hold under the weaker 
assumption that $\vp$ is \emph{semiproper} and generically submersive, as in Glaeser's theorem,
unless we admit the possibility of also shrinking $V$ to a suitable neighbourhood of $\vp^{-1}(0)$.
(In general, the conclusion of Theorem \ref{thm:contin}(2) holds for the restriction of
$\vp$ to some neighbourhood of any compact $K\subset \vp^{-1}(0)$.)
Corollary \ref{cor:estimate} nevertheless holds for $\vp$
semiproper because $f$ is assumed to be a composite $g\circ\vp$, where $g$ is $\cC^\infty$.
\end{remark}

\subsection{Model theory}\label{subsec:model}

\begin{remark}\label{rem:model}
To prove Theorem \ref{thm:definable}, we use the fact that, for a given quasianalytic class $\cQ$,
any closed \emph{sub-quasianalytic subset} $X$ of $\IR^n$ (the analogue in class $\cQ$ of a subanalytic set)
is the image of a proper $\cQ$-mapping $\s: Z \to \IR^n$, where $Z$ is a $\cQ$-manifold of dimension $= \dim X$.
This can be proved in the same way as the proof in \cite{BMihes} for a subanalytic set: First, $X$ is (at least
locally) the image by a proper $\cQ$-mapping of a \emph{semi-quasianalytic} subset $Y$ of $\IR^q$, for some $q$
(essentially by definition). Moreover, we can assume that $\dim Y = \dim X$, by the \emph{fibre-cutting lemma}
\cite[Lemma\,3.6]{BMihes}. It is easy to see that a closed semi-quasianalytic set is a proper image of a closed
\emph{quasianalytic} set (the $\cQ$-analogue of an analytic set) of the same dimension \cite[Prop.\,3.12]{BMihes},
and the result then follows from resolution of singularities.
\end{remark}

\begin{proof}[Proof of Theorem \ref{thm:definable}]
Let $B^n \subset \IR^n$ denote the closed unit ball, and define $\rho: \IR^n\to B^n$ by
$$
\rho(x) := \frac{x}{\sqrt{1 + \|x\|^2}},\quad x \in \IR^n,
$$
where $\|x \|^2 := x_1^2 +\cdots x_n^2$. Then $\rho$ is an analytic isomorphism onto the open ball.
Let $g := f\circ\rho^{-1}: \rho(W) \to \IR \subset S^1$ (where $S^1$ is the compactification of $\IR$ by
adding a single point at infinity), and let $X$ denote the closure of the graph of $g$ in 
$B^n\times S^1 \subset \IR^n\times S^1$. Since $f$ is definable, $X$ is a compact sub-quasianalytic subset
of $\IR^n \times S^1$ (of class $\cQ_M$).

By Remark \ref{rem:model}, there is a compact $\cQ_M$-manifold $Z$, where $\dim Z = \dim X = n$,
and a $\cQ_M$-mapping $\s:Z\to \IR^n\times S^1$ such that $\s(Z)=X$. Write $\s = (\s_1,\s_2)$ with respect
to the projections to $\IR^n$ and $S^1$. Then $g\circ\s_1 = \s_2$ on 
$\s^{-1}(\rho(W)\times S^1)\subset Z$, so that
$g\circ\s_1$ is of class $\cQ_M$ on $\s^{-1}(\rho(W)\times S^1)$. 
By Corollary \ref{cor:estimate}, $g\in \cQ_{M^{(p)}}(\rho(W))$,
for some $p$. Therefore, $f\in \cQ_{M^{(p)}}(W)$, since $f=g\circ\rho$ and $\rho$ is an 
analytic isomorphism from $\IR^n$ to the open ball.
\end{proof}

\section{Regularity estimates}\label{sec:estimate}

\begin{proof}[Proof of Theorem \ref{thm:estimate}]
We can assume that $K = [-r,r]^n \subset \mathbb{R}^n$ for some $r>0$. Let $\Jac\s$ denote the Jacobian
matrix $(\p \s/\p x)$, and write $\det \Jac \s = x^\ga \De(x)$, where $\De$ is nowhere vanishing.
Let $T = (T_{ji}) := (\Jac \s)^*/ \De(x)$, where $(\Jac\s)^*$ denotes the adjugate matrix of $\Jac\s$.
Let $(\p_{x_i})$ denote the (column) vector with entries $\p_{x_i}:= \p/\p x_i$, $i=1,\ldots,n$. Thus,
\begin{equation} \label{eq:trans}
(\p_{x_i}) = \Jac \s \cdot (\p_{y_j})\quad\text{ or }\quad (\p_{y_j}) = \frac{1}{x^\ga}\cdot T\cdot (\p_{x_i}).
\end{equation}
Since $f$ and $\s$ are of class $\cQ_M$, there is a constant $B\geq 1$ such that
\begin{equation}\label{eq:impl1}
\begin{aligned}
|f^{(\alpha)}| &\leq B^{|\alpha|} \alpha! M_{|\alpha|} \\
|T_{ji}^{(\alpha)}| &\leq B^{|\alpha|} \alpha! M_{|\alpha|+|\gamma|} ,\quad i, j = 1,\ldots, n,
\end{aligned}
\end{equation}
on $K$, for all $\al \in \IN^n$. 

Set $p := 2|\gamma|+1$. We will show that 
\begin{equation}\label{eq:impl2}
\left| \frac{\p^{|\be|} g}{\p y^{\be}}\right|  \leq (n^2 p \xi E^p )^{|\beta|} \beta! \,  M_{p|\beta|}
\end{equation}
on $\s(K)$, for all $\be \in \IN^n$, where $\xi$ is a fixed positive constant and $E>\max\{1,B\}$ is big enough that
\[
\sum_{\lambda \in \mathbb{N}^n} \left(\frac{B}{E}\right)^{|\lambda|} < \xi
\]
The estimate \eqref{eq:DCest} follows. In the argument below, 
all estimates are understood to mean ``on $K$'' or ``on $\s(K)$'',
as the case may be, and we will not say this explicitly.

\begin{claim}\label{claim:impl}
For each $\beta\in \IN^n$, 
\begin{equation}\label{eq:impl}
\left| \frac{\p^{|\al|} (g^{(\be)}\circ\s)}{\p x^{\al}} \right| \leq   (n \xi)^{|\beta|} E^{p|\beta| + |\alpha|} \Gamma(\alpha,\beta) M_{p_{\beta}(\alpha)},
\end{equation}
for all $\al \in \IN^n$, where $p_{\be}(\al) :=  p|\be|+|\al|$ and
\begin{equation}\label{eq:gamma}
\Gamma(\alpha,\beta)= \alpha! \cdot \prod_{j=1}^{|\beta|} \max_{1\leq i\leq n}\left\{ \alpha_i + j (\gamma_i+1) \right\}.
\end{equation}
\end{claim}

Note that $p_{\be}(0) =  p|\be|$, and that
\begin{align*}
\Gamma(0,\be) = \prod_{j=1}^{|\beta|}\max_{1\leq i\leq n}\left\{ j (\gamma_i+1) \right\} \leq \prod_{j=1}^{|\beta|} pj
= p^{|\beta|}|\beta|! \leq (n p)^{|\beta|} \beta!
\end{align*}
Therefore, \eqref{eq:impl} in the case that $\al=0$ implies \eqref{eq:impl2}; 
i.e., the estimate \eqref{eq:DCest} follows from Claim \ref{claim:impl}.

We will prove the claim by induction on $|\beta|$. Note that $\Gamma(\al,0) = \al!$. The claim is therefore true when $\beta = 0$, because in this case \eqref{eq:impl} reduces to \eqref{eq:impl1} (recall that $E>B$). Fix a multiindex $\tbe$, where $|\tbe| > 0$.
By induction, we assume the claim holds for all $\be$ such that $|\be| < |\tbe|$. 

Without loss of generality, there exists $\be \in \IN^n$ such that $\tbe = \be + (1)$. By \eqref{eq:trans} and the 
fundamental theorem of calculus (used $|\gamma|$ times),
\[
\frac{\partial g^{(\beta)}}{\partial y_1}  \circ \s = \int_{[0,1]^{|\gamma|}} \sum_{i=1}^n  \frac{\partial^{|\gamma|}}{\partial x^{\gamma}} \left( T_{1i} \cdot \frac{ \partial g^{(\beta)}\circ \s}{\partial x_i} \right) \bigg(\prod_{j=1}^{\gamma_k} t_{k,j} x_k\bigg) Q_{0}(t) dt 
\]
where $t:= (t_1,\ldots,t_n)$, $t_k:= (t_{k,1},\ldots, t_{k,\gamma_k})$, and 
$Q_0(t):=\prod_{k=1}^n \prod_{j=1}^{\gamma_k} t_{k,j}^{\gamma_k-j}$. It follows that, for all $\al \in \IN^n$, 
\[
\begin{aligned}
\left|\frac{\p^{|\al|}(g^{(\tbe)}\circ\s)}{\p x^{\al}}\right|\,  &= \left|\frac{\partial^{|\alpha|}}{\partial x^{\alpha}}\left(\frac{\partial g^{(\beta)}}{\partial y_1}  \circ \s \right)\right|\\
& \leq  \int_{[0,1]^{|\gamma|}} \sum_{i=1}^n \left|\frac{\partial^{|\gamma|+|\alpha|}}{\partial x^{\gamma+\alpha}} \left( T_{1i} \cdot \frac{ \partial g^{(\beta)}\circ \s}{\partial x_i} \right)
 \right| Q_{\alpha}(t) dt \\ 
&\leq \int_{[0,1]^{|\gamma|}} \sum_{i=1}^n \sum_{\lambda +\delta= \alpha+\gamma} \binom{\alpha+\gamma}{\lambda} \left| \frac{\partial^{|\lambda|} T_{1i}}{\partial x^{\lambda}} \cdot \frac{\partial^{|\delta|+1} g^{(\beta)} \circ \s}{\partial x^{\delta }\partial x_i}\right| Q_{\alpha}(t) dt 
\end{aligned}
\]
where $Q_{\alpha}(t) = \prod_{k=1}^n \prod_{j=1}^{\gamma_k} t_{k,j}^{\gamma_k+\alpha_k-j}$, so that:
\begin{equation}\label{eq:Q}
\int_{[0,1]^{|\gamma|}} Q_{\alpha}(t) dt = \frac{\alpha!}{(\alpha+\gamma)!} \,.
\end{equation}

Moreover, by \eqref{eq:impl1} and \eqref{eq:impl},
\[
\begin{aligned}
\left| \frac{\partial^{|\lambda|} T_{k,i}}{\partial x^{\lambda}}\right| & \leq B^{|\lambda|} \lambda! M_{|\lambda|+|\gamma|}\\
\left| \frac{\partial^{|\delta|+1} g^{(\beta)} \circ \s}{\partial x^{\delta }\partial x_i}\right| & \leq (n \xi)^{|\beta|} E^{p|\beta| + |\delta|+1} \Gamma(\delta+(i),\beta) M_{p_{\beta}(\delta+(i))}
\end{aligned}
\]
Since $p_{\beta}(\delta+(i)) + |\lambda|+|\gamma| = p|\beta| +2|\gamma|+1 + |\alpha| = p_{\tbe}(\alpha)$, 
the preceding inequalities imply (using log-convexity) that
\begin{equation}\label{eq:impl3}
\begin{aligned}
\left| \frac{\partial^{|\lambda|} T_{k,i}}{\partial x^{\lambda}}\cdot \frac{\partial^{|\delta|+1} g^{(\beta)} \circ \s}{\partial x^{\delta }\partial x_i}\right| &\leq (n \xi)^{|\beta|} B^{|\lambda|}E^{p|\beta| + |\delta|+1} \lambda!\, \Gamma(\delta+(i),\beta)  M_{p_{\tbe}(\alpha)}\\
&\leq  (n \xi)^{|\beta|} E^{p|\tbe|+|\alpha|} \left(\frac{B}{E}\right)^{|\lambda|}  \lambda! \,  \Gamma(\delta+(i),\beta)M_{p_{\tbe}(\alpha)}\,.
\end{aligned}
\end{equation}
We conclude from $\eqref{eq:Q}$ and $\eqref{eq:impl3}$ that
\[
\begin{aligned}
\left|\frac{\p^{|\al|}(g^{(\tbe)}\circ\s)}{\p x^{\al}}\right|\,  &\leq  \sum_{i=1}^n \sum_{\lambda +\delta = \alpha+\gamma} \binom{\alpha+\gamma}{\lambda} \frac{\alpha!\, \lambda!\,  \Gamma(\delta+(i),\beta)}{(\alpha+\gamma)!}  (n \xi)^{|\beta|} E^{p|\tbe|+|\al|}\left(\frac{B}{E}\right)^{|\lambda|}   M_{p_{\tbe}(\alpha)} \\
&\leq  \xi^{|\be|}(nE^p)^{|\tbe|}E^{|\alpha|} M_{p_{\tbe}(\alpha)} \sum_{\lambda +\delta = \alpha+\gamma} \frac{\alpha !}{\delta!}\max_{1\leq i\leq n} \{ \Gamma(\delta+(i),\beta)\} \left(\frac{B}{E}\right)^{|\lambda|}\,.
\end{aligned}
\]

Now, from the definition of $\Gamma(\alpha,\beta)$ and the fact that $\delta \leq \alpha+\gamma$, we obtain
\[
\begin{aligned}
\frac{\alpha ! }{\delta!} \max_{1\leq i\leq n} \{ \Gamma(\delta+(i),\beta)\} &= \alpha!
 \max_{1\leq i \leq n}\left\{ \frac{(\delta + (i))!}{\delta!} \prod_{j=1}^{|\beta|} \max_{1\leq k \leq n}\{ \delta_k +(1)+ j (\gamma_k+1)\} \right\} \\
 & \leq \alpha!
 \max_{1\leq i \leq n}\left\{ \frac{(\delta + (i))!}{\delta!} \prod_{j=1}^{|\beta|} \max_{1\leq k \leq n}\{ \alpha_k + \gamma_k+1+ j (\gamma_k+1)\} \right\} \\
 &\leq \alpha!
 \max_{1\leq i \leq n}\left\{ \frac{(\delta + (i))!}{\delta!} \right\}  \prod_{j=2}^{|\tbe|} \max_{1\leq i \leq n}\{ \alpha_i + j (\gamma_i+1)\}\\
 &\leq \alpha!
 \max_{1\leq i \leq n}\left\{ \alpha_i+\gamma_i+1 \right\}  \prod_{j=2}^{|\tbe|} \max_{1\leq i \leq n}\{ \alpha_i + j (\gamma_i+1)\}\\
 &= \alpha!
  \prod_{j=1}^{|\tbe|} \max_{1\leq i \leq n}\{ \alpha_i + j (\gamma_i+1)\} = \Gamma(\alpha,\tbe).
\end{aligned}
\]
We conclude, therefore, that
\[
\begin{aligned}
\left|\frac{\p^{|\al|}(g^{(\tbe)}\circ\s)}{\p x^{\al}}\right|\,  &\leq  \xi^{|\be|}(nE^p)^{|\tbe|} E^{|\al|}\Gamma(\alpha,\tbe) M_{p_{\tbe}(\alpha)} \sum_{\lambda \in \mathbb{N}^n} \left(\frac{B}{E}\right)^{|\lambda|}  \\
&\leq (n \xi E^p)^{|\tbe|} E^{|\alpha|} \Gamma(\alpha,\tbe) M_{p_{\tbe}(\alpha)},
\end{aligned}
\]
as required for the estimate \eqref{eq:DCest}.

It remains to show that we can get the estimate \eqref{eq:quasianDCest} if we make the stronger assumption
that $\cQ_M$ is closed under differentiation (i.e., $\cQ_M$ satisfies the assumption (a) of \S\ref{subsec:DC}). 
Under this assumption, we can strengthen the second inequality of \eqref{eq:impl1} to
$$
|T_{ji}^{(\alpha)}| \leq B^{|\alpha|} \alpha! M_{|\alpha|} ,\quad i, j = 1,\ldots, n,
$$
and the required estimate \eqref{eq:quasianDCest} can be obtained exactly as above.
\end{proof}

\begin{remark}
It is possible to prove Theorem \ref{thm:estimate} also by an argument similar to that of Chaumat and Chollet
\cite[Sect.\,III]{CC}. Indeed, in the notation of \cite{CC} (for a reader familiar with the latter), the estimate in
\cite[Lemma III.4]{CC} can be improved to
\[
\left| \frac{\partial}{\partial x^S}\left( T_{K,L}\right)(x) \right| \leq C_3^{l-s+1} \frac{(s+l)!}{k!} M_{ps}
\]
where $p=2|\ga|+1$, and this can be used in a calculation similar to that of \cite{CC}.
\end{remark}

\begin{proof}[Proof of Corollary \ref{cor:estimate}]
If $\dim V= \dim W$, then, by resolution of singularities, there is a proper surjective mapping $\s: Z\to V$ of class
$\cQ_M$ ($\dim Z = \dim V$) such that $\det (\p (\vp\circ\s)/\p x)$ is locally a monomial times an invertible factor (as in the proof
of Theorem \ref{thm:contin}), so the result follows from Theorem \ref{thm:estimate}.

In general, we can reduce to the preceding special case by using a mapping analogous to \eqref{eq:ext},
locally.
\end{proof}

\bibliographystyle{amsplain}

\end{document}